\documentclass[11pt]{article}

\usepackage[margin=1in]{geometry}
\usepackage{amsmath,amssymb,amsthm}
\usepackage[hidelinks]{hyperref}

\newtheorem{theorem}{Theorem}[section]
\newtheorem{lemma}[theorem]{Lemma}
\newtheorem{corollary}[theorem]{Corollary}
\numberwithin{equation}{section}

\newcommand{\N}{\mathbb N}
\newcommand{\Z}{\mathbb Z}

\title{Density bounds for permutations avoiding monotone arithmetic progressions}
\author{Jesse Geneson}
\date{}

\begin{document}

\maketitle

\begin{abstract}
For $X\in\{\N,\Z\}$, let $\alpha_X(\ell)$ and $\beta_X(\ell)$ denote the
supremal upper and lower densities of subsets of $X$ admitting
$\omega$-permutations without monotone $\ell$-term arithmetic progressions.
We strengthen the published lower bounds for the three-term upper-density parameters
by proving
\[
 \alpha_{\N}(3)\geq\frac23,\qquad
 \alpha_{\Z}(3)\geq\frac23.
\]
We also prove $\beta_{\Z}(4)=1$ by constructing four-permutable subsets of
the integers whose lower symmetric densities approach one.
\end{abstract}



\section{Introduction}

Throughout the paper, $\N=\{1,2,\ldots\}$.  An \emph{$\omega$-permutation}
of a countably infinite set $S$ is a sequence $p_0,p_1,\ldots$ of elements
of $S$ in which every element of $S$ occurs exactly once.  It contains a
monotone arithmetic progression of length $\ell$ if there are
$i_0<i_1<\cdots<i_{\ell-1}$, $x\in\Z$, and $d\in\Z\setminus\{0\}$ such that
\[
                         p_{i_j}=x+jd
                         \qquad(0\leq j<\ell).
\]
We call $S$ $\ell$-permutable if it has an $\omega$-permutation containing
no such progression.

For $S\subseteq\N$ and $T\subseteq\Z$, define
\[
\begin{aligned}
 \overline d_{\N}(S)&=\limsup_{n\to\infty}
       \frac{|S\cap[1,n]|}{n},&
 \underline d_{\N}(S)&=\liminf_{n\to\infty}
       \frac{|S\cap[1,n]|}{n},\\
 \overline d_{\Z}(T)&=\limsup_{n\to\infty}
       \frac{|T\cap[-n,n]|}{2n+1},&
 \underline d_{\Z}(T)&=\liminf_{n\to\infty}
       \frac{|T\cap[-n,n]|}{2n+1}.
\end{aligned}
\]
For $X\in\{\N,\Z\}$, put
\[
\begin{aligned}
 \alpha_X(\ell)&=\sup\{\overline d_X(S):S\subseteq X
       \text{ is infinite and $\ell$-permutable}\},\\
 \beta_X(\ell)&=\sup\{\underline d_X(S):S\subseteq X
       \text{ is infinite and $\ell$-permutable}\}.
\end{aligned}
\]
These parameters require an order of type $\omega$.  A doubly infinite
permutation indexed by $\Z$ or an arbitrary linear order is a different
object.  Ardal, Brown, and Jungi\'c
\cite[Theorem~2.2]{ArdalBrownJungic} constructed a linear order on all of
$\Z$ containing no monotone three-term progression.

Entringer and Jackson \cite{EntringerJackson} asked whether, for every $n$,
the set $\{0,1,\ldots,n\}$ has a permutation containing no monotone
three-term arithmetic progression.  Davis, Entringer, Graham, and Simmons
\cite[Facts~3 and 4]{DavisEtAl} proved that every $\omega$-permutation of
$\N$ contains a monotone three-term progression and constructed one
containing no monotone five-term progression.  They asked for the supremal
upper and lower densities of three-permutable subsets of $\N$ and posed
analogous questions for $\Z$ \cite[Questions~4 and 5]{DavisEtAl}.

LeSaulnier and Vijay
\cite[Theorem~3 and p.~207]{LeSaulnierVijay} proved
$\alpha_{\N}(3)\geq1/2$ and $\beta_{\N}(3)\geq1/4$, and conjectured
equality in both bounds.  Geneson
\cite[Introduction and Proposition~4]{Geneson} defined
the two-sided parameters and proved $\alpha_{\Z}(3)\geq1/2$ and
$\beta_{\Z}(3)\geq1/6$.  Adenwalla \cite[Theorem~8]{Adenwalla} improved
the latter bound to $\beta_{\Z}(3)\geq3/10$.

For length four, LeSaulnier and Vijay
\cite[Theorem~3]{LeSaulnierVijay} proved $\alpha_{\N}(4)=1$ and
$\beta_{\N}(4)\geq1/3$, Geneson \cite[Proposition~3]{Geneson} improved
the latter bound to $1/2$, and Adenwalla
\cite[Theorems~6 and 7]{Adenwalla} proved
$\beta_{\N}(4)=\alpha_{\Z}(4)=1$ and $\beta_{\Z}(4)\geq2/3$.
Adenwalla \cite[Question~3]{Adenwalla} asked for the correct values of the
four length-three parameters and $\beta_{\Z}(4)$.  The published results
and the results here are as follows.
\[
\begin{array}{c|c|c}
\text{parameter}&\text{published result}&\text{result here}\\ \hline
\alpha_{\N}(3)&\geq1/2&\geq2/3\\
\alpha_{\Z}(3)&\geq1/2&\geq2/3\\
\beta_{\Z}(4)&\geq2/3&=1
\end{array}
\]

\begin{theorem}\label{thm:three-bounds}
\[
 \alpha_{\N}(3)\geq\frac23,\qquad
 \alpha_{\Z}(3)\geq\frac23.
\]
\end{theorem}

The first inequality disproves the equality conjecture of
LeSaulnier and Vijay. The next construction determines the remaining length-four parameter.

\begin{theorem}\label{thm:four-family}
Let $m\geq2$ be a power of two, let $a\geq4$ be an integer, and let
$c\in\mathbb R$ satisfy
\[
                         2+\frac1a<c<a.
\]
There is a $4$-permutable set $S_{m,a,c}\subseteq\Z$ such that
\[
 \underline d_{\Z}(S_{m,a,c})
       =1-\frac{a(c-1)}{2mc(a-1)}.
\]
\end{theorem}

\begin{corollary}\label{cor:all-four}
For every integer $\ell\geq4$,
\[
 \alpha_{\N}(\ell)=\beta_{\N}(\ell)
 =\alpha_{\Z}(\ell)=\beta_{\Z}(\ell)=1.
\]
\end{corollary}

The equalities in Corollary~\ref{cor:all-four} are supremum statements.
Theorem~\ref{thm:four-family} does not produce a four-permutable subset of
the integers of lower symmetric density one or a four-progression-free
$\omega$-permutation of all of $\N$ or $\Z$.  The exact values of the four
length-three parameters remain open.

\section{Binary orders}

The binary recursion is classical; see Davis et al. \cite[p.~81]{DavisEtAl}
and Ardal, Brown, and Jungi\'c
\cite[Definition~2.1]{ArdalBrownJungic}.  Nathanson \cite{Nathanson} proved
the power-of-two modular existence theorem, and Ardal, Brown, and Jungi\'c
\cite[Definition~2.2]{ArdalBrownJungic} constructed the compatible order on
$\Z$.  We record the comparison identities needed below.

For each integer $t\geq0$, represent every residue modulo $2^t$ by its
unique member of $\{0,1,\ldots,2^t-1\}$ and its $t$-digit binary
expansion.  Let
$\bar\rho_{2^t}:\Z/2^t\Z\longrightarrow\{0,1,\ldots,2^t-1\}$ assign to
each residue its rank when the representative's digits are read from least
significant to most significant, with $0$ before $1$ at the first
difference, and put
$\rho_{2^t}(z)=\bar\rho_{2^t}(z+2^t\Z)$ for $z\in\Z$.  Thus the residue
order modulo $4$ is $0,2,1,3$.  These ranks satisfy
\begin{equation}\label{eq:binary-recurrence}
\rho_{2^{j+1}}(z)=2^j(z\bmod2)
       +\rho_{2^j}\left(\left\lfloor\frac z2\right\rfloor\right)
       \qquad(j\geq0).
\end{equation}
Here $z\bmod2\in\{0,1\}$ is the least significant binary digit, and
$\lfloor z/2\rfloor$ removes it.
For a nonzero integer $z$, let $\nu_2(z)$ be the largest $v\geq0$ for
which $2^v$ divides $z$.

\begin{lemma}\label{lem:binary-residues}
Let $t\geq1$ and $m=2^t$.  If $r,\delta\in\Z$ and
$\delta\not\equiv0\pmod m$, then
\begin{align}
 \rho_m(r)<\rho_m(r+\delta)
   &\quad\Longleftrightarrow\quad
 \rho_m(r+\delta)>\rho_m(r+2\delta),\label{eq:binary-zigzag}\\
 \rho_m(r)<\rho_m(r+\delta)
   &\quad\Longleftrightarrow\quad
 \rho_m(r+2\delta)<\rho_m(r+3\delta).
   \label{eq:binary-pairs}
\end{align}
In particular, $\rho_m(r+\delta)$ is strictly smaller than both
$\rho_m(r)$ and $\rho_m(r+2\delta)$, or strictly larger than both.
\end{lemma}

\begin{proof}
We induct on $t$.  For $t=1$, the integer $\delta$ is odd, and the four
ranks are $0,1,0,1$ when $r$ is even and $1,0,1,0$ when $r$ is odd, so
both equivalences hold.  Suppose the assertion holds modulo $2^t$ and
consider modulus $2^{t+1}$.  If
$\delta$ is odd, the parities of
$r,r+\delta,r+2\delta,r+3\delta$ alternate.  By
\eqref{eq:binary-recurrence}, every odd term has rank at least $2^t$ and
every even term has rank less than $2^t$.  If $r$ is even, each side of
both equivalences is true; if $r$ is odd, each side is false.  Thus both
equivalences hold.
If $\delta=2\delta'$, all four terms have the same parity and
\[
 \left\lfloor\frac{r+k\delta}{2}\right\rfloor
 =\left\lfloor\frac r2\right\rfloor+k\delta'
 \qquad(0\leq k\leq3).
\]
Moreover, $2^{t+1}\nmid\delta$ implies $2^t\nmid\delta'$.  For
$0\leq k\leq3$, the first summand in \eqref{eq:binary-recurrence} is
$2^t((r+k\delta)\bmod2)$.  Since $\delta$ is even, it is independent of
$k$.  Subtracting this common summand from each rank comparison reduces
both equivalences to the induction hypothesis for $\lfloor r/2\rfloor$
and $\delta'$.  Since $\delta\not\equiv0\pmod m$, the residues
$r,r+\delta$ are distinct, as are $r+\delta,r+2\delta$, so the two
comparisons in \eqref{eq:binary-zigzag} are strict.  If the first comparison
is $<$, the second is $>$; otherwise both inequalities reverse.  Hence the
middle rank is strictly larger than both endpoint ranks or strictly smaller
than both.
\end{proof}

For distinct $u,v\in\Z$, put $q=1+\nu_2(v-u)$ and declare
\[
 u\prec v\quad\Longleftrightarrow\quad
 \rho_{2^q}(u)<\rho_{2^q}(v).
\]
This defines a strict total order on $\Z$.  It is the order of Ardal,
Brown, and Jungi\'c
\cite[Definitions~2.1 and 2.2]{ArdalBrownJungic}.

\begin{lemma}\label{lem:integer-binary-order}
For $u,h\in\Z$ with $h\neq0$, neither of the relations
\[
 u\prec u+h\prec u+2h
 \qquad\text{or}\qquad
 u+2h\prec u+h\prec u
\]
holds, and
\[
             u\prec u+h\quad\Longleftrightarrow\quad
             u+2h\prec u+3h.
\]
\end{lemma}

\begin{proof}
The two forbidden chains follow from Ardal, Brown, and Jungi\'c
\cite[Theorem~2.2]{ArdalBrownJungic}.  The equivalence is the case $s=2$
and $t=3$ of Hirose and Saito \cite[Lemma~2.5]{HiroseSaito}, applied to
the identity map from $\Z$ to $(\Z,\prec)$.
\end{proof}

We call a finite order or an order of type $\omega$ \emph{three-AP-free} if
it contains no monotone three-term arithmetic progression.

\begin{corollary}\label{cor:finite-order}
Every finite set of integers has an order containing no monotone
three-term arithmetic progression.
\end{corollary}

This follows by restricting the order of Ardal, Brown, and Jungi\'c
\cite[Definition~2.2 and Theorem~2.2]{ArdalBrownJungic}; we use it for every
finite block below.

\section{Positive upper density}

\begin{lemma}\label{lem:positive-extension}
Let $M$ and $L$ be positive integers with $L>2M$, let
$A\subseteq[1,M]\cap\Z$, and let $J$ be a nonnegative integer.  Put
\begin{equation}\label{eq:positive-extension-block}
 B=\bigcup_{j=0}^{J}
       [L4^j+M,2L4^j]\cap\Z.
\end{equation}
No three-term arithmetic progression contained in $A\cup B$ meets both
$A$ and $B$.
\end{lemma}

\begin{proof}
Every member of $B$ is larger than $M$.  Suppose that two terms of a
progression lie in $A$ and one lies in $B$.  The term in $B$ must be the
largest term.  If the two terms in $A$ are $u<v$, the largest possible third
term is $2v-u\leq2M-1$, while the least member of $B$ is $L+M>3M$.  This is
impossible.

It remains to consider a progression with one term $a\in A$ and two terms
$y<z$ in $B$.  Then $a<y<z$ and
\begin{equation}\label{eq:positive-reflection}
                              z=2y-a.
\end{equation}
If $y$ belongs to the $j$th interval in
\eqref{eq:positive-extension-block}, then
\[
                    L4^j+M\leq y\leq2L4^j.
\]
Since $1\leq a\leq M$, equation \eqref{eq:positive-reflection} gives
\[
                    2L4^j+M\leq z\leq4L4^j-1.
\]
Thus $z$ exceeds the greatest member $2L4^j$ of the $j$th interval.  If
$j<J$, then $z$ is smaller than the least member $4L4^j+M$ of the
$(j+1)$st interval.  If $j=J$, then $z$ exceeds every member of $B$.
In either case $z\notin B$.
\end{proof}

\begin{lemma}\label{lem:predecessor-count}
Let $<$ be an infinite strict total order on a set $X$, and suppose that
every $x\in X$ has finitely many predecessors.  The map
\[
                 p(x)=|\{y\in X:y<x\}|
\]
is an order-preserving bijection from $X$ to $\N\cup\{0\}$.
\end{lemma}

\begin{proof}
If $x<y$, then every predecessor of $x$, and also $x$ itself, precedes
$y$.  Hence $p(x)<p(y)$.  Now fix $y$ with $p(y)=q$, and list its $q$
predecessors in increasing order as $y_0<\cdots<y_{q-1}$.  Every global
predecessor of $y_j$ also precedes $y$, so it is one of
$y_0,\ldots,y_{j-1}$; conversely, all these $j$ elements precede $y_j$.
Thus $p(y_j)=j$.  The image of $p$ is therefore an initial segment of
$\N\cup\{0\}$.  It is infinite because $X$ is infinite and $p$ is
injective, so it is all of $\N\cup\{0\}$.
\end{proof}

Put $S_0=\varnothing$ and $M_0=1$.  For $k\geq1$, define
\begin{equation}\label{eq:positive-upper-parameters}
                         L_k=4M_{k-1}
\end{equation}
and
\begin{equation}\label{eq:positive-upper-block}
 B_k=\bigcup_{j=0}^{k}
 [L_k4^j+M_{k-1},2L_k4^j]\cap\Z.
\end{equation}
Finally, put
\begin{equation}\label{eq:positive-upper-support}
 M_k=2L_k4^k,\qquad
 S_k=S_{k-1}\cup B_k,\qquad
 S=\bigcup_{k\geq1}B_k.
\end{equation}

Since $M_k=8M_{k-1}4^k>M_{k-1}$, induction gives
$S_k=\bigcup_{i=1}^kB_i\subseteq[1,M_k]$.  Every member of $B_k$ is larger
than $M_{k-1}$, so the blocks are pairwise disjoint.  Each block is finite
and nonempty, since its first interval contains $L_k+M_{k-1}$.  Hence $S$
is infinite.  By Corollary~\ref{cor:finite-order},
choose a three-AP-free
order $\sigma_k$ of the entire finite set $B_k$, and concatenate the
orders as
\begin{equation}\label{eq:positive-upper-order}
                         \sigma_1\sigma_2\sigma_3\cdots.
\end{equation}
This concatenation lists every member of $S$ exactly once, and every member
has only finitely many predecessors.  Lemma~\ref{lem:predecessor-count}
therefore makes it an $\omega$-permutation of $S$.

\begin{lemma}\label{lem:positive-upper-avoidance}
The $\omega$-permutation in \eqref{eq:positive-upper-order} is
three-AP-free.
\end{lemma}

\begin{proof}
A progression whose three terms belong to one $B_k$ is excluded by
$\sigma_k$.  Suppose that a progression meets more than one block, and
let $B_k$ be the block with largest index among its terms.  Its other term
or terms belong to $S_{k-1}\subseteq[1,M_{k-1}]$.  Lemma
\ref{lem:positive-extension}, with $A=S_{k-1}$, $M=M_{k-1}$,
$L=L_k$, and $J=k$, says that no arithmetic progression meets both sets.
This is a contradiction.
\end{proof}

For $0\leq j<k$, we have
$2L_k4^j<L_k4^{j+1}+M_{k-1}$, so the intervals in
\eqref{eq:positive-upper-block} are pairwise disjoint.  The $j$th interval has
\[
                         L_k4^j-M_{k-1}+1
\]
terms.  Hence
\begin{equation}\label{eq:positive-upper-size}
 |B_k|=L_k\frac{4^{k+1}-1}{3}
          -(k+1)(M_{k-1}-1).
\end{equation}
Since $B_k\subseteq[1,M_k]$ for $M_k=2L_k4^k$,
\begin{align*}
 \frac{|S\cap[1,M_k]|}{M_k}
 &\geq\frac{|B_k|}{M_k}\\
 &=\frac23-\frac{1}{6\cdot4^k}
   -\frac{(k+1)(M_{k-1}-1)}{8M_{k-1}4^k}\\
 &\geq\frac23-\frac{1}{6\cdot4^k}
                 -\frac{k+1}{8\cdot4^k}.
\end{align*}
Since $4^{-k}\to0$, $(k+1)4^{-k}\to0$, and $M_k\to\infty$, the displayed
lower bounds tend to $2/3$.  The definition of upper
density therefore gives
\[
                         \overline d_{\N}(S)\geq\frac23.
\]
Lemma~\ref{lem:positive-upper-avoidance} proves
$\alpha_{\N}(3)\geq2/3$.

\section{Two-sided upper density}

The signed construction permits progressions meeting more than one block
whose midpoint belongs to an earlier block.  Such a progression is not
monotone because its midpoint occurs before both endpoints.

\begin{lemma}\label{lem:signed-extension}
Let $M$ and $L$ be positive integers with $L>2M$, let
$A\subseteq[-M,M]\cap\Z$, and let $J$ be a nonnegative integer.  Put
\begin{equation}\label{eq:signed-extension-block}
 I_j=[L4^j+M,2L4^j-M]\cap\Z
 \quad(0\leq j\leq J),
 \qquad
 B=\bigcup_{j=0}^J(I_j\cup-I_j).
\end{equation}
Suppose that every element of $A$ is ordered before every element of $B$.
Every three-term arithmetic progression contained in $A\cup B$ and meeting
both sets has its midpoint in $A$ and both endpoints in $B$.  Hence it is
not monotone.
\end{lemma}

\begin{proof}
The least absolute value of a member of $B$ is $L+M>3M$.  If two terms of
a progression, say $u,v$, lie in $A$, then a third term which is their
midpoint has absolute value at most $M$.  A third term which is an endpoint
has the form $2u-v$, and hence has absolute value at most
$2|u|+|v|\leq3M$.  Thus no progression has two terms in $A$ and one in $B$.

Suppose that a progression has one term $a\in A$ and two terms in $B$.
If $a$ is its midpoint, this is the conclusion of the lemma.  Suppose
instead that $a$ is an endpoint.  After exchanging the endpoints if
necessary, write the progression as
\[
                         a,\quad y,\quad z,
                         \qquad z=2y-a.
\]
If $y>0$, then $z\geq2|y|-|a|>0$; if $y<0$, then
$z\leq-2|y|+|a|<0$.  Thus $y,z$ have the same sign, and the triangle
inequalities give
\[
                         2|y|-M\leq|z|\leq2|y|+M.
\]
If $|y|\in I_j$, then
\[
                  2L4^j+M\leq|z|\leq4L4^j-M.
\]
The lower bound is larger than the greatest member $2L4^j-M$ of $I_j$.
If $j<J$, the upper bound is smaller than the least member $4L4^j+M$ of
$I_{j+1}$.  If $j=J$, then $|z|$ exceeds the greatest magnitude in $B$.
Thus $z\notin B$, a contradiction.
\end{proof}

Put $T_0=\varnothing$ and $R_0=1$.  For $k\geq1$, define
\begin{equation}\label{eq:signed-upper-parameters}
                         L_k=4R_{k-1}
\end{equation}
and
\begin{equation}\label{eq:signed-upper-intervals}
 I_{k,j}=[L_k4^j+R_{k-1},2L_k4^j-R_{k-1}]\cap\Z
 \quad(0\leq j\leq k).
\end{equation}
Put
\begin{equation}\label{eq:signed-upper-support}
\begin{aligned}
 B_k&=\bigcup_{j=0}^k(I_{k,j}\cup-I_{k,j}),\\
 R_k&=2L_k4^k-R_{k-1},\\
 T_k&=T_{k-1}\cup B_k,
 \qquad T=\bigcup_{k\geq1}B_k.
\end{aligned}
\end{equation}

Since $R_k=(8\cdot4^k-1)R_{k-1}>R_{k-1}$, induction gives
$T_k=\bigcup_{i=1}^kB_i\subseteq[-R_k,R_k]$.  Every member of $B_k$ has
absolute value at least $5R_{k-1}$, so the blocks are pairwise disjoint, and
$R_k$ is the largest absolute value introduced in $B_k$.  Each block is
finite and nonempty, since $I_{k,0}$ contains $L_k+R_{k-1}$.  Hence $T$ is
infinite.  Give the whole set $B_k$ a
finite three-AP-free order and concatenate these orders in increasing order
of $k$.

\begin{lemma}\label{lem:signed-upper-avoidance}
This concatenation is a three-AP-free $\omega$-permutation of $T$.
\end{lemma}

\begin{proof}
The concatenation lists every member of $T$ exactly once, and every member
has finitely many predecessors, so Lemma~\ref{lem:predecessor-count} makes
it an $\omega$-permutation.  A
progression contained in one $B_k$ is excluded by its internal order.  For a
progression meeting more than one block, let $k$ be the largest index for
which it has a term in $B_k$.  Lemma
\ref{lem:signed-extension} applies to $A=T_{k-1}$ and $B=B_k$.  If the
progression has one term in $B_k$, the lemma says that it cannot exist.  If
it has two terms in $B_k$, its remaining term is its midpoint in
$T_{k-1}$ and occurs before both endpoints.  It is therefore not monotone.
\end{proof}

For $0\leq j<k$,
$2L_k4^j-R_{k-1}<L_k4^{j+1}+R_{k-1}$, so the positive intervals are
pairwise disjoint.  They contain only positive integers, while their
negatives contain only negative integers.  Thus all sets in the union are
disjoint.  Each positive interval in \eqref{eq:signed-upper-intervals} has
\[
                         L_k4^j-2R_{k-1}+1
\]
terms.  Hence
\begin{equation}\label{eq:signed-upper-size}
 |B_k|=2L_k\frac{4^{k+1}-1}{3}
       -2(k+1)(2R_{k-1}-1).
\end{equation}
Put $U_k=L_k4^k$, so that $R_k=2U_k-R_{k-1}$.  Since
$B_k\subseteq[-R_k,R_k]$, keeping only its contribution gives
\begin{equation}\label{eq:signed-upper-density}
 \frac{|T\cap[-R_k,R_k]|}{2R_k+1}
 \geq
 \frac{\frac83U_k-\frac23L_k-4(k+1)R_{k-1}+2(k+1)}
      {4U_k-2R_{k-1}+1}.
\end{equation}
Since $L_k=4R_{k-1}$,
\[
 \frac{L_k}{U_k}=4^{-k},
 \qquad
\frac{R_{k-1}}{U_k}=4^{-(k+1)}.
\]
Also $U_k=L_k4^k\geq4^{k+1}$, so $(k+1)/U_k\to0$.  Dividing the numerator
and denominator in \eqref{eq:signed-upper-density} by $U_k$, the numerator
tends to $8/3$ and the denominator to $4$.  Since $R_k\to\infty$, the
definition of upper density gives
\[
                         \overline d_{\Z}(T)\geq\frac23.
\]
Lemma~\ref{lem:signed-upper-avoidance} proves
$\alpha_{\Z}(3)\geq2/3$.

\section{Four-term progressions}

The construction combines the binary order of Section~2 with Adenwalla's
geometric residue blocks, truncation, and alternating signed intervals
\cite[Theorems~1, 4, 6, and 7]{Adenwalla}.

Fix a power of two $m\geq2$, an integer $a\geq4$, and a real
number $c$ satisfying
\begin{equation}\label{eq:four-parameters}
                         2+\frac1a<c<a.
\end{equation}
Write $\rho=\rho_m$.  For
$0\leq i<m$, let
$r_i\in\{0,1,\ldots,m-1\}$ be the residue with $\rho(r_i)=i$.  For
$s\geq0$, define
\begin{equation}\label{eq:four-blocks}
\begin{aligned}
 N_{s,i}={}&\left\{x<0:x\equiv r_i\pmod m,
       \ a^{2ms+m-i}\leq |x|<\frac{a^{2ms+3m-i}}c\right\},\\
 P_{s,i}={}&\left\{x>0:x\equiv r_i\pmod m,
       \ a^{2ms+m+1+i}\leq x<\frac{a^{2ms+3m+1+i}}c\right\}.
\end{aligned}
\end{equation}
Order the blocks at stage $s$ as
\begin{equation}\label{eq:four-block-order}
 N_{s,0},N_{s,1},\ldots,N_{s,m-1},
 P_{s,m-1},P_{s,m-2},\ldots,P_{s,0},
\end{equation}
and concatenate the stages for $s=0,1,2,\ldots$.  For $x,y$ in a block
with residue $r$, put $x$ before $y$ if
\[
 \frac{x-r}{m}\prec\frac{y-r}{m}
 \quad\text{in a negative block},\qquad
 \frac{y-r}{m}\prec\frac{x-r}{m}
 \quad\text{in a positive block}.
\]

We call the exponent of a block's lower magnitude
endpoint its \emph{lower exponent}. For a fixed residue and sign, consecutive stages have lower exponents
$f$ and $f+2m$.  Since $c>1$, the first block ends before
$a^{f+2m}$, where the next begins, so these blocks are disjoint.  Blocks
with different residues or signs contain different integers.  Every block is
finite.  If $f$ is its lower exponent, then
$a^{f+2m}/c>a^{f+2m-1}$, so its defining magnitude interval contains
$[a^f,a^{f+2m-1})$.  Its length satisfies
\[
 a^f(a^{2m-1}-1)\geq4(4^m-1)\geq4m>m.
\]
For $P_{s,i}$, the first positive integer at least $a^f$ that is congruent
to $r_i$ modulo $m$ is smaller than $a^f+m$ and hence lies in this
interval.  For $N_{s,i}$, take instead the first positive magnitude in the
residue class $-r_i$ modulo $m$ and negate it; the resulting negative
integer is congruent to $r_i$ modulo $m$.  Thus every block is nonempty.
Let
$S_{m,a,c}$ be the union of the blocks.
There are infinitely many pairwise disjoint nonempty blocks, so
$S_{m,a,c}$ is infinite.
Every member lies in exactly one finite block, and each block has only
finitely many predecessor stages and blocks.  Thus every member has finitely
many predecessors.  Lemma~\ref{lem:predecessor-count} identifies this
infinite order with $\N\cup\{0\}$, so it is an $\omega$-permutation
$\pi_{m,a,c}$ of $S_{m,a,c}$.

Assign each block a rank $q(B)\in\{0,\ldots,m-1\}$ by
\begin{equation}\label{eq:four-block-ranks}
                 q(N_{s,i})=i,\qquad q(P_{s,i})=m-1-i.
\end{equation}
Every $x\in S_{m,a,c}$ lies in a unique block $B(x)$; put
$q(x)=q(B(x))$.  In either sign, the lower exponent in
\eqref{eq:four-blocks} has the form
\begin{equation}\label{eq:four-exponent}
                         2ms+K-q,
\end{equation}
where $K=m$ for negative blocks and $K=2m$ for positive blocks.

If an arithmetic progression has common difference $d$ with $m\nmid d$,
every adjacent pair has distinct residues modulo $m$, hence distinct
$\rho$-ranks.  For a same-sign pair, the block ranks are also distinct,
because they use either the $\rho$-order or its reverse.

For a finite arithmetic progression $x_0,\ldots,x_v$, call the block
sequence $B(x_0),\ldots,B(x_v)$ chronologically nondecreasing if its blocks
occur in nondecreasing order in the block sequence defining
$\pi_{m,a,c}$.

If a block has lower exponent $f$, its \emph{outer exponent} is $f+2m$
and its strict outer magnitude endpoint is
\begin{equation}\label{eq:four-outer-endpoint}
                         U_f=\frac{a^{f+2m}}c.
\end{equation}

\begin{lemma}\label{lem:four-scale}
Let two distinct blocks of one sign occur in chronological order, with
ranks $q$ and $q'$, respectively, where $q'\leq q$.  If the first block
has outer endpoint $U_f$, then every magnitude in the later block is at
least
\[
                         ca^{q-q'}U_f.
\]
In particular, it is more than $c$ times every magnitude in the first
block, and more than $ac$ times every magnitude in the first block when
$q'<q$.
\end{lemma}

\begin{proof}
Among the negative blocks or among the positive blocks of a fixed stage,
the block ranks increase chronologically.  Thus the later block belongs
to a later stage.  By
\eqref{eq:four-exponent}, its lower exponent in the earliest possible
stage is $f+2m+q-q'$.  Its magnitudes are therefore at least
$a^{f+2m+q-q'}=ca^{q-q'}U_f$.  A still later stage only increases this
bound, and every magnitude in the first block is strictly less than $U_f$.
\end{proof}

Call a same-sign segment of an arithmetic progression \emph{shrinking} if
its magnitudes decrease and \emph{growing} if its magnitudes increase.

\begin{lemma}\label{lem:four-shrinking}
A chronologically nondecreasing same-sign shrinking three-term arithmetic
progression is contained in one block.
\end{lemma}

\begin{proof}
Let $d$ be its common difference.  If $m\mid d$, all three terms have the
same block rank.  Lemma~\ref{lem:four-scale} prevents a transition to a
later block because the magnitudes decrease.  If $m\nmid d$, the two
successive block-rank comparisons are opposite by
\eqref{eq:binary-zigzag}, whether the ranks use $\rho$ or its reverse.
Adjacent terms have different residues and lie in distinct blocks.  At
one transition the later block has smaller rank, so
Lemma~\ref{lem:four-scale} makes its magnitudes larger, contradicting
shrinking.
\end{proof}

\begin{lemma}\label{lem:four-crossing}
Let $x_0,x_1,x_2$ be consecutive terms of an arithmetic progression with
signs $-,+,+$ or $+,-,-$.  Suppose that the last pair grows in magnitude,
the three blocks are chronologically nondecreasing, and
$q(x_2)\leq q(x_1)$.  Then $x_1$ and $x_2$ lie in the same block.
\end{lemma}

\begin{proof}
Let the block containing $x_1$ belong to stage $s$, have lower exponent
$f$, and have outer endpoint $U_f$.  Every chronologically earlier block
of the opposite sign has outer endpoint at most $U_f/a$.  If $x_1$ is
positive, the largest outer exponent of an earlier negative block in the
same stage is $2ms+3m$, while the outer exponent of the positive block is
at least $2ms+3m+1$.  If $x_1$ is negative, every earlier positive block
belongs to a preceding stage; its outer exponent is at most $2ms+2m$,
while that of the negative block is at least $2ms+2m+1$.

The block containing $x_0$ is one of these earlier opposite-sign blocks,
so $|x_0|<U_f/a$, while $|x_1|<U_f$.  It follows that
\[
                  |x_2|=2|x_1|+|x_0|
                    <\left(2+\frac1a\right)U_f<cU_f.
\]
If $x_2$ belonged to a distinct later block with nonincreasing block rank,
Lemma~\ref{lem:four-scale} would give $|x_2|\geq cU_f$, a contradiction.
\end{proof}

If $z_{j-1},z_j,z_{j+1}$ are consecutive same-sign terms of an arithmetic
progression with common difference $d$ and their magnitudes grow, then
\begin{equation}\label{eq:four-growing-ratio}
 |z_{j+1}|=|z_j|+|d|<2|z_j|,
 \qquad\text{since}\qquad |z_j|=|z_{j-1}|+|d|>|d|.
\end{equation}

\begin{lemma}\label{lem:four-growing}
The following statements hold.
\begin{enumerate}
\item Let $y,x_0,x_1,x_2$ be consecutive terms of an arithmetic
progression whose blocks are chronologically nondecreasing.  If
$x_0,x_1,x_2$ have one sign, grow in magnitude, and have sign opposite to
$y$, then $x_0,x_1,x_2$ lie in one block.
\item If $x_0,x_1,x_2,x_3$ is a chronologically nondecreasing growing
four-term arithmetic progression of one sign, then $x_1,x_2,x_3$ lie in
one block.
\end{enumerate}
\end{lemma}

\begin{proof}
Let $d$ be the common difference.  If $m\mid d$, the three same-sign terms
in the first statement have one block rank.  Applied to $y,x_0,x_1$,
Lemma~\ref{lem:four-crossing} puts $x_0,x_1$ in one block.  Equation
\eqref{eq:four-growing-ratio} gives $|x_2|<2|x_1|$.  Since $c>2$,
Lemma~\ref{lem:four-scale} prevents $x_2$ from lying in a distinct later
block of the same rank.  In the second statement, all four terms have one
block rank.  Applying \eqref{eq:four-growing-ratio} to the triples
$x_0,x_1,x_2$ and $x_1,x_2,x_3$, and then applying
Lemma~\ref{lem:four-scale}, puts the last three terms in one block.

Suppose that $m\nmid d$.  For the first statement,
\eqref{eq:binary-zigzag} makes the later block have strictly smaller rank on
one of the two transitions, because adjacent terms have different residues.
If this occurs from $x_0$ to $x_1$, Lemma
\ref{lem:four-crossing} puts two terms with different residues in one
block, a contradiction.  If it occurs from $x_1$ to $x_2$, equation
\eqref{eq:four-growing-ratio} gives $|x_2|<2|x_1|$, whereas
Lemma~\ref{lem:four-scale} gives $|x_2|>ac|x_1|>2|x_1|$, a contradiction.
For the second statement, the block-rank comparisons from
$x_1$ to $x_2$ and from $x_2$ to $x_3$ are opposite by
\eqref{eq:binary-zigzag}.  At the strict descent, the two terms have
different residues, so Lemma~\ref{lem:four-scale} makes the later magnitude
more than $ac>2$ times the earlier one, contrary to
\eqref{eq:four-growing-ratio}.
\end{proof}

\begin{lemma}\label{lem:four-block-pattern}
Let $x_0,x_1,x_2,x_3$ be a nonconstant four-term arithmetic progression
in $S_{m,a,c}$ whose blocks are chronologically nondecreasing.  Then either
three consecutive terms lie in one block, or $x_0,x_1$ lie in one block
$A$ and $x_2,x_3$ lie in one block $B$, where $A$ and $B$ have the same
residue modulo $m$ and opposite signs.
\end{lemma}

\begin{proof}
No term is zero.  Since the progression is nonconstant, its terms are
strictly monotone as integers.  If all four have one sign, their magnitudes
either shrink or grow.  Lemma~\ref{lem:four-shrinking} or the second part
of Lemma~\ref{lem:four-growing} puts three consecutive terms in one block.
Otherwise the sign changes once, and the two maximal consecutive
same-sign segments have lengths $1$ and $3$, $2$ and $2$, or $3$ and $1$.
Indeed, a strictly monotone integer sequence not containing zero changes
sign at most once; before the change its magnitudes shrink, and after the
change they grow.  In the $3+1$ case, Lemma~\ref{lem:four-shrinking} puts
the initial three terms in one block.  In the $1+3$ case, the first part of
Lemma~\ref{lem:four-growing}, with the initial term as $y$, puts the terminal
three terms in one block.  Only the $2+2$ case remains.

Let $d$ be the common difference.  If $m\mid d$, the shrinking pair has
one block rank, so Lemma~\ref{lem:four-scale} puts it in one block.  Lemma
\ref{lem:four-crossing} puts the growing pair in one block.  These blocks
have the same residue and opposite signs.

Suppose that $m\nmid d$.  The shrinking pair must have increasing block
ranks, since Lemma~\ref{lem:four-scale} precludes a nonincreasing block
rank across a shrinking transition.  The ranks are unequal because
adjacent residues differ.  If the shrinking pair is negative,
then $\rho(x_0)<\rho(x_1)$.  Equation \eqref{eq:binary-pairs} gives
$\rho(x_2)<\rho(x_3)$, strictly, so $q(x_2)>q(x_3)$ because positive blocks use the
reverse residue order.  If the shrinking pair is positive, then
$\rho(x_0)>\rho(x_1)$, so \eqref{eq:binary-pairs} gives
$\rho(x_2)>\rho(x_3)$, strictly.  Hence $q(x_2)>q(x_3)$ for the negative growing
pair.  In either case, Lemma~\ref{lem:four-crossing} would put the growing
pair in one block although its residues are different.  Thus $m\nmid d$
cannot occur when two terms have each sign.
\end{proof}

\begin{lemma}\label{lem:four-avoidance}
The $\omega$-permutation $\pi_{m,a,c}$ contains no monotone four-term
arithmetic progression.
\end{lemma}

\begin{proof}
Suppose that $x_0,x_1,x_2,x_3$ is such a progression in its order of
appearance.  Its blocks are chronologically nondecreasing.  By
Lemma~\ref{lem:four-block-pattern}, three consecutive terms lie in one
block, or the progression has the two-block form stated there.

In the first case, the three terms have a common residue $r$ modulo $m$.
The map $x\mapsto(x-r)/m$ gives a nonconstant three-term arithmetic
progression that occurs monotonically in $\prec$ or its reverse, contrary
to Lemma~\ref{lem:integer-binary-order}.

In the second case, let $r$ be the common residue and put
$z_i=(x_i-r)/m$.  If $h=z_1-z_0$, then $h\neq0$ and
\[
                     (z_2,z_3)=(z_0+2h,z_1+2h).
\]
Indeed, $x_{i+1}-x_i$ is independent of $i$ and equals $mh$, so
$z_{i+1}-z_i=h$ for $0\leq i<3$.
Lemma~\ref{lem:integer-binary-order} gives
\begin{equation}\label{eq:four-two-block-comparison}
                     z_0\prec z_1
                     \quad\Longleftrightarrow\quad z_2\prec z_3.
\end{equation}
If $A$ is negative and $B$ is positive, their internal orders require
$z_0\prec z_1$ and $z_3\prec z_2$, contrary to
\eqref{eq:four-two-block-comparison}.  If $A$ is positive and $B$ is
negative, they require $z_1\prec z_0$ and $z_2\prec z_3$, again a
contradiction.
\end{proof}

Together with the construction above,
  Lemma~\ref{lem:four-avoidance} establishes the four-permutability of
  $S_{m,a,c}$.  To complete the proof of
  Theorem~\ref{thm:four-family}, it remains to compute its lower symmetric
  density.  We do this by describing the integers omitted outside a fixed
  central interval, namely the magnitude gaps between consecutive blocks
  of each fixed sign and residue, and computing the upper symmetric density
  of the complement.

A block with lower exponent $f$ occupies the magnitude interval
\[
                       [a^f,a^{f+2m}/c).
\]
The block with the same residue and sign in the next stage begins at
$a^{f+2m}$.  Its intervening gap is
\begin{equation}\label{eq:four-gap}
                         [a^e/c,a^e),
 \qquad e=f+2m.
\end{equation}
The lower exponents of the blocks at stage $s$ are
\begin{equation}\label{eq:four-stage-exponents}
                 \{2ms+1,2ms+2,\ldots,2ms+2m\}.
\end{equation}
Each exponent in \eqref{eq:four-stage-exponents} occurs once.  The gap
exponents arising from stage $s$ are therefore
\[
                 \{2ms+2m+1,\ldots,2ms+4m\}.
\]
The last of these is one less than the first gap exponent for stage
$s+1$.  Thus every integer $e\geq e_0:=2m+1$ is the right endpoint
exponent of exactly one gap.

Since $c<a$, gaps with consecutive right endpoint exponents are disjoint:
the gap ending at $a^e$ lies below $a^e<a^{e+1}/c$, the lower endpoint of
the next gap.  We next verify that there are no other omissions.  Fix a sign
and a residue class.  The lower exponents of its blocks have the form
$f_0+2ms$, where $1\leq f_0\leq2m$ and $s\geq0$.  Let $x$ have the fixed
sign and residue and satisfy $|x|>a^{2m}$.  Choose the unique integer
$E\geq2m$ such that
\[
                         a^E\leq|x|<a^{E+1}.
\]
Since $f_0\leq2m\leq E$, let $f$ be the largest number of the form
$f_0+2ms$ with $f\leq E$.
Then $f\leq E<f+2m$.  If $|x|<a^{f+2m}/c$, then $x$ belongs to the block
with lower exponent $f$.  Otherwise,
\[
             \frac{a^{f+2m}}c\leq|x|<a^{E+1}\leq a^{f+2m},
\]
so $x$ belongs to that block's following gap.  Consequently, outside
$[-a^{2m},a^{2m}]$, each missing integer lies in exactly one gap, and each
gap removes exactly the integers of one sign in one residue class.

For an integer $n\geq0$, let
\[
 H(n)=|[-n,n]\cap(\Z\setminus S_{m,a,c})|,
 \qquad
 \lambda=\frac{1-1/c}{m},
 \qquad
 L=\frac{\lambda a}{2(a-1)}.
\]
For a gap with right endpoint $a^e$, put
\[
                    \ell_e=\left\lceil\frac{a^e}{c}\right\rceil.
\]
After replacing missing integers by their positive magnitudes, the gap
consists exactly of the integers in $[\ell_e,a^e)$ in one residue class
modulo $m$.  Let $h_e$ be their number.  For integer endpoints $u\leq v$,
the number of integers in one residue class modulo $m$ that belong to
$[u,v)$ differs from $(v-u)/m$ by less than $1$.  Since
$0\leq\ell_e-a^e/c<1$,
\begin{align}
 |h_e-\lambda a^e|
 &\leq\left|h_e-\frac{a^e-\ell_e}{m}\right|
       +\frac{\ell_e-a^e/c}{m}\notag\\
 &<1+\frac1m\leq2.\label{eq:four-gap-error}
\end{align}
In particular,
  \[
     |h_e-\lambda a^e|\leq2.
  \]
Put $C=2a^{2m}+1$, the number of integers in
$[-a^{2m},a^{2m}]$.
For each integer $K\geq0$, put $n_K=a^{e_0+K}$.  The gaps with endpoint
exponents $e_0,\ldots,e_0+K$ lie in $[-n_K,n_K]$.  The next gap begins at
$a^{e_0+K+1}/c>n_K$, so every remaining missing integer in this interval
lies in $[-a^{2m},a^{2m}]$.  Hence
$0\leq H(n_K)-\sum_{e=e_0}^{e_0+K}h_e\leq C$, and
\[
\begin{split}
 \left|H(n_K)-\lambda\sum_{e=e_0}^{e_0+K}a^e\right|
 &\leq
 \left|H(n_K)-\sum_{e=e_0}^{e_0+K}h_e\right|
 +\sum_{e=e_0}^{e_0+K}|h_e-\lambda a^e|\\
 &\leq C+2(K+1).
\end{split}
\]
Since
\[
 \sum_{e=e_0}^{e_0+K}a^e
 =\frac{a^{e_0+K+1}-a^{e_0}}{a-1},
 \qquad n_K=a^{e_0+K},
\]
the normalized geometric main term is
\[
 \frac{\lambda\sum_{e=e_0}^{e_0+K}a^e}{2n_K+1}
 =\frac{\lambda}{a-1}
   \frac{a-a^{-K}}{2+a^{-(e_0+K)}}
 \longrightarrow\frac{\lambda a}{2(a-1)}=L.
\]
Also $(C+2(K+1))/(2n_K+1)\to0$, since $n_K=a^{e_0+K}$.
Dividing the preceding error estimate by $2n_K+1$ therefore gives
\[
 \lim_{K\to\infty}\frac{H(n_K)}{2n_K+1}=L.
\]

Now let $n\geq a^{2m}$, and let $E\geq2m$ be the unique integer such that
$a^E\leq n<a^{E+1}$.  Set $X=a^{E+1}$ and $y=n+1$.
Outside $[-a^{2m},a^{2m}]$, every missing integer of magnitude at most $n$
belongs to a gap with right endpoint at most $a^E$ or to the part of
$[X/c,X)$ below $n$.  By \eqref{eq:four-gap-error}, the completed gaps contain at most
  \[
   \lambda\sum_{e=e_0}^{E}a^e+2E
   \leq\frac{\lambda X}{a-1}+2E
  \]
  integers, where the sum is zero if $E<e_0$.

  Write $\ell_{E+1}=\lceil X/c\rceil$.  If $y\leq\ell_{E+1}$, the
  truncated part of the gap is empty.  If $y>\ell_{E+1}$, it consists of
  the integers in $[\ell_{E+1},y)$ in one residue class modulo $m$.
  The residue-class estimate and $\ell_{E+1}\geq X/c$ give
  \[
   \#\{\text{integers in the truncated part}\}
   <\frac{y-\ell_{E+1}}m+1
   \leq\frac{y-X/c}m+1.
  \]
  Thus in both cases the truncated part contains at most
  \[
               \max\left\{0,\frac{y-X/c}m\right\}+1
  \]
  integers. Hence
 \[
   H(n)\leq
   \frac{\lambda X}{a-1}
   +\max\left\{0,\frac{y-X/c}{m}\right\}
   +C+2E+1.
  \]
If $y\leq X/c$, the maximum is zero, and $X\leq ay$ gives
$\lambda X/(a-1)\leq\lambda ay/(a-1)$.  If $y>X/c$, then
$\lambda=(c-1)/(mc)$, $X\geq y$, and $c-a<0$ give
\[
\frac{\lambda X}{a-1}
 +\frac{y-X/c}{m}
 =\frac ym+\frac{c-a}{mc(a-1)}X
 \leq\frac ym+\frac{c-a}{mc(a-1)}y
 =\frac{\lambda ay}{a-1}.
\]
Thus the first two terms on the right side of the bound for $H(n)$ sum to
at most $\lambda ay/(a-1)=2Ly$.  Since $a^E\leq n$ and $E/a^E\to0$, we have
$0\leq E/n\leq E/a^E\to0$.  Dividing the bound for $H(n)$ by $2n+1$,
the term $2Ly/(2n+1)$ tends to $L$, while
 $(C+2E+1)/(2n+1)$ tends to zero.  It follows that
\[
 \limsup_{n\to\infty}\frac{H(n)}{2n+1}\leq L.
\]
The subsequence $n=n_K$ gives the reverse inequality, so
\[
 \limsup_{n\to\infty}\frac{H(n)}{2n+1}
 =\frac{a(c-1)}{2mc(a-1)}.
\]
Since $|S_{m,a,c}\cap[-n,n]|/(2n+1)=1-H(n)/(2n+1)$, taking the lower
limit and using $\liminf(1-u_n)=1-\limsup u_n$ gives
\[
 \underline d_{\Z}(S_{m,a,c})
       =1-\frac{a(c-1)}{2mc(a-1)}.
\]
Together with Lemma~\ref{lem:four-avoidance}, this proves
Theorem~\ref{thm:four-family}.

\begin{proof}[Proof of Corollary~\ref{cor:all-four}]
Take $a=4$ and $c=5/2$.  These values satisfy
\eqref{eq:four-parameters}, and Theorem~\ref{thm:four-family} gives, for
every power of two $m\geq2$,
\[
                         \beta_{\Z}(4)\geq1-\frac2{5m}.
\]
Letting the powers of two $m$ tend to infinity gives
$\beta_{\Z}(4)=1$, since every density is at most one.

Adenwalla \cite[Theorem~6]{Adenwalla} proved $\beta_{\N}(4)=1$.

Upper density is at least lower density.  A monotone progression of length
$\ell\geq4$ contains one of length four, so every four-permutable set is
$\ell$-permutable.  Thus $\beta_X(4)=1$ implies
$\beta_X(\ell)=\alpha_X(\ell)=1$ for $X\in\{\N,\Z\}$, because all four
parameters are at most one.
\end{proof}

\section*{Declaration of generative AI and AI-assisted technologies}

During the preparation of this draft, the author used Codex with GPT-5.6 Sol Ultra to assist with proof development, literature searches, and manuscript editing. The author reviewed and edited the content as needed and takes full responsibility for the content of the article.


\begin{thebibliography}{99}
\setlength{\itemsep}{0pt}

\bibitem{Adenwalla}
S. Adenwalla,
Avoiding monotone arithmetic progressions in permutations of integers,
\emph{Discrete Math.} \textbf{347} (2024), 114183.
\href{https://doi.org/10.1016/j.disc.2024.114183}
{doi:10.1016/j.disc.2024.114183}.

\bibitem{ArdalBrownJungic}
H. Ardal, T. Brown, and V. Jungi\'c,
Chaotic orderings of the rationals and reals,
\emph{Amer. Math. Monthly} \textbf{118} (2011), 921--925.
\href{https://doi.org/10.4169/amer.math.monthly.118.10.921}
{doi:10.4169/amer.math.monthly.118.10.921}.

\bibitem{DavisEtAl}
J. A. Davis, R. C. Entringer, R. L. Graham, and G. J. Simmons,
On permutations containing no long arithmetic progressions,
\emph{Acta Arith.} \textbf{34} (1977), 81--90.
\href{https://doi.org/10.4064/aa-34-1-81-90}
{doi:10.4064/aa-34-1-81-90}.

\bibitem{EntringerJackson}
R. C. Entringer and D. E. Jackson,
Elementary problem E2440,
\emph{Amer. Math. Monthly} \textbf{80} (1973), no.~9, 1058.
\href{https://doi.org/10.2307/2318789}{doi:10.2307/2318789}.

\bibitem{Geneson}
J. Geneson,
Forbidden arithmetic progressions in permutations of subsets of the
integers,
\emph{Discrete Math.} \textbf{342} (2019), 1489--1491.
\href{https://doi.org/10.1016/j.disc.2019.02.004}
{doi:10.1016/j.disc.2019.02.004}.

\bibitem{HiroseSaito}
M. Hirose and S. Saito,
Characterization of order structures avoiding three-term arithmetic
progressions,
\emph{Order} \textbf{42} (2025), 231--239.
\href{https://doi.org/10.1007/s11083-024-09677-7}
{doi:10.1007/s11083-024-09677-7}.

\bibitem{LeSaulnierVijay}
T. D. LeSaulnier and S. Vijay,
On permutations avoiding arithmetic progressions,
\emph{Discrete Math.} \textbf{311} (2011), 205--207.
\href{https://doi.org/10.1016/j.disc.2010.10.006}
{doi:10.1016/j.disc.2010.10.006}.

\bibitem{Nathanson}
M. B. Nathanson,
Permutations, periodicity, and chaos,
\emph{J. Combin. Theory Ser. A} \textbf{22} (1977), 61--68.
\href{https://doi.org/10.1016/0097-3165(77)90063-2}
{doi:10.1016/0097-3165(77)90063-2}.

\end{thebibliography}
\end{document}